\documentclass[reqno, 12pt]{amsart}
\usepackage{verbatim}
\usepackage{amssymb}
\usepackage{enumerate}
\usepackage[active]{srcltx}
   \usepackage{geometry}
\numberwithin{equation}{section}

\usepackage{t1enc}
\usepackage[latin2]{inputenc}

\newtheorem{theorem}{Theorem}[section]

\newtheorem{corollary}[theorem]{Corollary}

\theoremstyle{definition}

\theoremstyle{proposition}

\newcommand{\mc}[1]{\mathcal{#1}}

\newcommand{\setm}{\setminus}

\newcommand{\subs}{\subset}

\def\<{\left\langle}
\def\>{\right\rangle}
\def\br#1;#2;{\bigl[ {#1} \bigr]^ {#2} }

\newcommand{\dis}{\mc D}

\author[A. Dow]{Alan Dow}
\address{UNC Charlotte}
\email{adow@uncc.edu}

\author[I. Juh\'asz]{Istv\'an Juh\'asz}
\address      { Alfréd Rényi Institute of Mathematics, Hungarian Academy of Sciences }
\email{juhasz@renyi.hu}

\author[L. Soukup]{Lajos Soukup}
\thanks
  {
   }
\address
      { Alfr{\'e}d R{\'e}nyi Institute of Mathematics, Hungarian Academy of Sciences }
\email{soukup@renyi.hu}

\author[Z. Szentmikl\'ossy]{Zolt\'an Szentmikl\'ossy}
\address{E\"otv\"os University of Budapest}
\email{szentmiklossyz@gmail.com}

\author[W. Weiss]{William Weiss}
\address{University of Toronto}
\email{weiss@math.utoronto.ca}

\subjclass[2010]{54A25, 03E35, 54A35}
\keywords{tightness, Frechet, $G_\delta$-modification, compact, Lindelöf, pseudocharacter,
non-reflecting stationary set, strongly compact cardinal}

\title[$G_\delta$-tightness]{On the tightness of $G_\delta$-modifications}
\date{\today}

\begin{document}

\maketitle

\begin{abstract}
The $G_\delta$-modification $X_\delta$ of a topological space $X$ is the space on the same underlying set
generated by, i.e. having as a basis, the collection of all $G_\delta$ subsets of $X$. Bella and Spadaro recently investigated
in \cite{BS} the connection between the values of various cardinal functions taken on $X$ and
$X_\delta$, respectively. In their paper, as Question 2, they raised the following problem:
Is $t(X_\delta) \le 2^{t(X)}$ true for every (compact) $T_2$ space $X$? Note that this is actually two
questions.

In this note we answer both questions: In the compact case affirmatively and in the
non-compact case negatively. In fact, in the latter case we even show that it is consistent with ZFC that no upper bound
exists for the tightness of the $G_\delta$-modifications of countably tight, even Frechet spaces.
\end{abstract}

\section{Introduction}

The $G_\delta$-modification $X_\delta$ of a topological space $X$ is the space on the same underlying set
generated by, i.e. having as a basis, the collection of all $G_\delta$ subsets of $X$.
Bella and Spadaro recently investigated
in \cite{BS} the connection between the values of various cardinal functions taken on $X$ and
$X_\delta$, respectively. In their paper, as Question 2, they raised the following problem:
Is $t(X_\delta) \le 2^{t(X)}$ true for every (compact) $T_2$ space $X$? Note that this is actually two
questions.
In this note we answer both questions: In the compact case affirmatively and in the non-compact case negatively.
Actually, for the compact case we prove something stronger: We show that for every regular
Lindelöf space $X$ we have $t(X_\delta) \le 2^{t(X)}$. In the non-compact case we shall show that it is consistent with ZFC that no upper bound
exists for the tightness of the $G_\delta$-modifications of countably tight, even Frechet spaces.

We shall use standard notation and terminology from set theory and general topology. In particular,
concerning cardinal functions, we follow the notation and terminology of \cite{J}
It will be useful to denote by $G_\delta(X)$ the family of all $G_\delta$ subsets of a space $X$.
So, as we said, $G_\delta(X)$ is a basis for $X_\delta$.
For any $A \subs X$ we shall use $\overline{A}^\delta$ to denote the closure of $A$ in $X_\delta$.

Just like in \cite{BS}, our proofs will often use elementary submodels of appropriate "initial segments"
of the form $H(\lambda)$ of the universe. Most readers will be
familiar enough with these notions and for those who are not they are surveyed e.g. in \cite{D}.

\section{Bounds for the tightness of $G_\delta$-modifications}

\begin{theorem}\label{tm:LCT}
If $X$ is a regular Lindelöf space then $t(X_\delta) \le 2^{t(X)}$.
\end{theorem}

\begin{proof}
Assume that $X$ is a regular Lindelöf space, $p \in X$ and $A \subs X$ are such that $p \in \overline{A}^\delta$.
Let us put $\kappa = 2^{t(X)}$ and choose an elementary submodel $M$ of an appropriate $H(\lambda)$ such that
$|M| = \kappa, \, M^{t(X)} \subs M$, moreover $\{X, A, p\} \subs M$. We shall show that then $p \in \overline{A \cap M}^\delta$,
which by $|A \cap M| \le \kappa$ will complete our proof.

To see this, assume that $p \in H \in G_\delta(X)$, i.e. $H = \bigcap_{n < \omega} V_n$ where each $V_n$ is open in $X$.
We have to prove that $H \cap A \cap M \ne \emptyset$.

For every point $x \in \overline{A \cap M}$ (note that this is the closure in $X$) there is a subset
$B_x \subs A \cap M$ with $|B_x| \le t(X)$ such that $x \in \overline{B_x}$. Then $M^{t(X)} \subs M$ implies $B_x \in M$.
Clearly, if $x \ne p$ then $B_x$ can be chosen so that $p \notin \overline{B_x}$. In this case, by the regularity of $X$,
there are open sets $U_x \supset \overline{B_x}$ and $W_x$ with $p \in W_x$ such that $U_x \cap W_x = \emptyset$, moreover
as both $\overline{B_x}$ and $p$ belong to $M$, we may assume that $U_x$ and $W_x$ also belong to $M$.

Now, the closed subspace $\overline{A \cap M} \setm V_n$ of $X$ is Lindelöf for each $n < \omega$,
hence there is a countable subset $C_n \subs \overline{A \cap M} \setm V_n$ such that
$\overline{A \cap M} \setm V_n \subs \bigcup_{x \in C_n} U_x$. This clearly implies that if we put $W_n = \bigcap_{x \in C_n} W_x$ then
$$A \cap M \cap W_n \subs V_n \,.$$
Although we do not know if $C_n \in M$, we do know that $W_n \in M$ because $W_x \in M$ for each $x \in C_n$
and $M$ is countably closed.

It follows then that $W = \bigcap_{n < \omega} W_n \in M \cap G_\delta(X)$, hence $p \in \overline{A}^\delta$ and $p \in W$ imply
$W \cap A \ne \emptyset$ and, by elementarity, $W \cap A \cap M \ne \emptyset$ as well. But then we have
$\emptyset \ne W \cap A \cap M \subs H \cap A \cap M$, which completes our proof.

\end{proof}

The one-point compactification of an uncountable discrete space has countable tightness, it is even Frechet,
and its $G_\delta$-modification clearly has tightness $\omega_1$. This, of course, shows that Theorem \ref{tm:LCT}
is sharp for countably tight compact spaces under CH. But what happens if the continuum $\mathfrak{c}$ is large?
Actually, we do not know the full answer to this question.

However, we happen to have a ready made consistent answer in \cite{CH*} where a weakening of CH called CH* was
introduced. It was shown there that CH* holds in any model obtained by adding any number of Cohen reals to a
ground model that satisfies CH. Thus CH* is consistent with $\mathfrak{c}$ being anything it can be.

Let us denote by $\ell_{\omega_1}(A)$ the set of all points obtainable as the limit of a converging $\omega_1$-sequence
of points of $A$ in a space $X$. We permit constant sequences, hence $A \subs \ell_{\omega_1}(A)$.
It is obvious that we always have $\ell_{\omega_1}(A) \subs \overline{A}^\delta$. Now, for countably tight compacta,
the proof of Theorem 3.2 of \cite{CH*} actually establishes the following converse of this.

\begin{theorem}\label{tm:CCT}
CH* implies that if $X$ is any countably tight compactum and $A \subs X$ then $\ell_{\omega_1}(A) \supset \overline{A}^\delta$,
hence $\ell_{\omega_1}(A) = \overline{A}^\delta$. Consequently, if $X_\delta$ is non-discrete then  $t(X_\delta) = \omega_1$.
\end{theorem}

Although the statement of Theorem 3.2 in \cite{CH*} is slightly weaker than this, the reader may easily check
that actually this is proved there.

This result then leads us to the following natural and intriguing  question.

\medskip

PROBLEM 1.
Is it consistent to have a countably tight compactum $X$ for which $t(X_\delta) > \omega_1$?

\medskip

It turns out by our next two ZFC results that the, somewhat surprising, equality $\ell_{\omega_1}(A) = \overline{A}^\delta$
may occur in other situations as well. It will be useful to introduce the following notation:
If $X$ is a space and ${\kappa}$ is a cardinal then we write
\begin{displaymath}
 \dis_{\kappa}(X)=\{D\in \br X;{\kappa}; : \text{$D$ is discrete}\}.
\end{displaymath}
A countably tight compact, or just Lindelöf space $X$ contains no uncountable free sequences,
i.e. satisfies $F(X) = \omega$. This makes the assumption $F(X) = \omega$ in our following result fitting with
the topic of this paper.

\begin{theorem}\label{tm:psi=w}
Let $X$ be a regular space such that $F(X) = \omega$ and for every $D \in \dis_\omega(X)$
we have $\psi(\overline{D}) \le \omega$. Then for every $A \subs X$ we have $\,\ell_{\omega_1}(A) = \overline{A}^\delta$.
\end{theorem}

\begin{proof}
Assume that $p \in X$ and $A \subs X$ are such that $p \in \overline{A}^\delta \setm A$.
By induction on ${\alpha}<{\omega}_1$ we shall define closed $G_{\delta}$ sets $H_{\alpha}$ containing $p$
and points $x_{\alpha}\in A\cap H_{\alpha}$ as follows.

Assume that $\{H_{\beta}:{\beta}<{\alpha}\}$ and $Y_{\alpha}=\{x_{\beta}:{\beta}<{\alpha}\}$
have been defined, moreover $Y_{\alpha}$ is a free sequence in $X\setm \{p\}$, hence $Y_{\alpha} \in \dis_{\omega}(X)$.
Then either $p \notin \overline {Y_{\alpha}}$ or $\psi(p,\overline {Y_{\alpha}})\le {\omega}$.
But in both cases there is a closed $G_{\delta}$ set
$H$ containing $p$ such that $H\cap (\{p\}\cup \overline {Y_{\alpha}}) = \{p\}$.
We then let $H_{\alpha}=H\cap\bigcap_{{\beta}<{\alpha}}H_{\beta} \in G_\delta(X)$  and use
$p \in \overline{A}^\delta$ to pick the point $x_\alpha \in A \cap H_\alpha$.

Thus we have constructed $\{x_{\beta}:{\beta}<{\omega}_1\} \subs A$.
The sequence  $\{x_{\beta}:{\beta}<{\omega}_1\}$ is free in $X\setm \{p\}$ because
for every ${\alpha}<{\omega}_1$ we have $\overline{Y_{\alpha}}\cap H_{\alpha} \subs \{p\}$ and
$\overline{\{x_{\beta}:{\beta}\ge {\alpha}\}}\subs H_{\alpha}$.

If  $U$ is any open set containing $p$ then $\{x_{\alpha}:x_{\alpha}\notin U\}$ is free in $X$, hence it is countable.
In other words, $U$ contains a tail of $\{x_{\alpha}:{\alpha}<{\omega}_1\}$, i.e. the $\omega_1$-sequence
$\{x_{\alpha}:{\alpha}<{\omega}_1\} \subs A$ indeed converges to $x$.
\end{proof}

Clearly, the condition $\psi(\overline{D}) \le w(\overline{D}) = \omega$ is satisfied for
any countable subset $D$ of the $\Sigma$-product $\Sigma(\kappa)$ taken inside the Tychonov cube of weight $\kappa$.
Also, compact subspaces of such $\Sigma$-products, i.e. Corson-compacta are Frechet, hence do not contain
uncountable free sequences. Thus we immediately obtain the following corollary of Theorem \ref{tm:psi=w}:

\begin{corollary}\label{Corson}
For every subset $A$ of a Corson-compact space $X$ we have $\,\ell_{\omega_1}(A) = \overline{A}^\delta$.
\end{corollary}

To facilitate the formulation of our next result, we introduce the notation $CAP(\kappa)$ to denote the class
of all spaces in which every subset of cardinality $\kappa$ has a complete accumulation point.

\begin{theorem}\label{tm:psi=w1}
Assume that $X \in CAP(\omega_1)$ is a countably tight regular space such that $\psi(\overline{S}) \le \omega_1$
for every countable subset $S \subs X$. Then for every $A \subs X$ we have $\,\ell_{\omega_1}(A) = \overline{A}^\delta$.
\end{theorem}

\begin{proof}
Consider any point $p \in \overline{A}^\delta \setm A$ and then choose an $\omega_1$-chain $\<N_{\alpha}:{\alpha}<{\omega}_1\>$ of countable elementary submodels
of an appropriate $H(\lambda)$  such that\\
\smallskip
(i) $\{X, A, p\} \subs N_0$;\\
\smallskip
(ii) for every $\beta < \omega_1$ we have $\<N_{\alpha}:{\alpha}<{\beta}\>\in N_{\beta}$.
\smallskip
Let $N=\bigcup_{{\alpha}<{\omega}_1}N_{\alpha}$.

Since $X$ has countable tightness, for every $\alpha < \omega_1$ we have $p \in \overline{N_{\alpha} \cap A}$,
hence  $\psi(p,\overline{N_{\alpha}\cap A})\le {\omega}_1$. It follows
that there is a family $\mc U_{\alpha} \in N_{{\alpha}+1}$ of open sets with $|\mc U_{\alpha}| \le \omega_1$ such that
 \begin{displaymath}
  \{p\}=\bigcap \mc U_{\alpha}\cap \overline{N_{\alpha}\cap A}.
 \end{displaymath}
Note that then $\mc U_{\alpha} \in N_{{\alpha}+1} \subs N$ and $|\mc U_{\alpha}|\le {\omega}_1$ imply $\mc U_{\alpha}\subs N$.
Consequently, for all $\alpha < \omega_1$ we have
\begin{displaymath}
 \{p\}=\bigcap \{U\in N\cap \tau(X): p\in U   \}\cap \overline{N_{\alpha}\cap A},
\end{displaymath}
where $\tau(X)$ denotes the topology of $X$.
Since $X$ has countable tightness, we also have $\overline{N\cap A}=\bigcup_{{\alpha}<{\omega}_1}\overline{N_\alpha \cap A}$,
hence
\begin{displaymath}
 \{p\}=\bigcap \{U\in N\cap \tau(X) : p\in U   \}\cap \overline{N\cap A}.
\end{displaymath}

Let us now put
\begin{displaymath}
H_{\alpha}=\bigcap \{U\in N_{\alpha}\cap \tau(X) : p\in U   \}.
\end{displaymath}
Then $H_\alpha$ is a $G_\delta$ set that we claim is closed in $X$.
Indeed, this is because for every $U\in N_\alpha\cap \tau(X)$ with $p\in U$ there is, by the regularity of $X$ and by elementarity,
some $V\in N_\alpha\cap \tau(X)$ with $p\in V$ such that $\overline{V} \subs U$.

Since $N\cap \tau(X) = \bigcup_{\alpha < \omega_1} N_\alpha\cap \tau(X)$, it follows that
\begin{displaymath}
 \{p\}=\bigcap_{\alpha < \omega_1}H_\alpha \cap \overline{N\cap A}.
\end{displaymath}

Now $p \in \overline{A}^\delta$ and
$p \in H_{\alpha} \in N$ imply that for every $\alpha < \omega_1$ we can pick a point
\begin{displaymath}
 x_{\alpha}\in N\cap A\cap H_{\alpha}.
\end{displaymath}

We claim that the sequence $S = \{x_{\alpha}:{\alpha}<{\omega}_1\} \subs A$  converges to $p$.
Indeed, if $q\ne p$ then there is a $\beta < \omega_1$ with $q \notin H_\beta \cap \overline{N\cap A}$.
Then $X \setm H_\beta \cap \overline{N\cap A}$ is a neighborhood of $q$ that misses the final segment
$\{x_{\alpha} : \beta \le {\alpha}<{\omega}_1\}$, hence $q$ is not a complete accumulation point of $S$.
But $X \in CAP(\omega_1)$ then implies that $p$ is the unique complete accumulation point of $S$,
and hence $S$ indeed converges to $p$.

\end{proof}

\section{A large cardinal bound for the tightness of the $G_\delta$-modifications}

The first result of this section shows not only that the answer to the original question of Bella and Spadaro
in the non-compact (or non-Lindelöf) case is negative, in fact it shows that no reasonable bound exists,
at least in ZFC.

\begin{theorem}\label{tm:nob}
Assume that $S$ is a non-reflecting stationary set of $\omega$-limits in an uncountable regular cardinal $\kappa$.
Then there is a 0-dimensional Frechet topology $\tau$ on $\kappa + 1 = \kappa \cup \{\kappa\}$ such that for the space $X = (\kappa + 1, \tau)$
we have $t(X_\delta) = \kappa$.
\end{theorem}

\begin{proof}
Let us denote by $\mathcal{V}$ the family of all subsets $V$ of $\kappa$ having the property that for every
$\alpha \in S$ there is $\beta < \alpha$ with $(\beta, \alpha) = \alpha \setm \beta \subs V$. We define $\tau$
to be the topology on $\kappa + 1$ for which all points in $\kappa$ are isolated and $\{V \cup \{\kappa\} : V \in  \mathcal{V}\}$
forms a neighborhood base for the point $\kappa$. More precisely,
$\tau = \mathcal{P}(\kappa) \cup \{V \cup \{\kappa\} : V \in  \mathcal{V}\}$.
It is simple to verify that $\tau$ is indeed a 0-dimensional $T_2$ topology.

To see that $\tau$ is Frechet, observe that if $A \subs \kappa$ accumulates to the point $\kappa$ then there is
an $\alpha \in S$ such that $\sup (A \cap \alpha) = \alpha$. But then there is an increasing $\omega$-sequence
$B = \{\beta_n : n < \omega\} \subs A \cap \alpha$  with $\sup B = \alpha$ and clearly $B$ converges to the point $\kappa$.

Since $S$ is stationary, it is an immediate consequence of Fodor's theorem that every set $V \in  \mathcal{V}$
includes a final segment of $\kappa$. By $cf(\kappa) > \omega$ it follows then that every $G_\delta$ set
containing the point $\kappa$ also includes a final segment of $\kappa$, hence the set $\kappa$ accumulates to
the point $\kappa$ in the space $X_\delta$. Thus, since $\kappa$ is regular, to prove $t(X_\delta) = \kappa$ it will suffice to show that
no proper initial segment of $\kappa$ accumulates to the point $\kappa$ in the space $X_\delta$.

This, of course, is equivalent to showing that for each $\eta < \kappa$ the initial segment $\eta$ is an
$F_\sigma$ set in $X$. We do this by transfinite induction on $\eta < \kappa$. Thus assume that $\eta < \kappa$
and we know that for every $\zeta < \eta$ the initial segment $\zeta$ is an
$F_\sigma$ set $X$, i.e. $\zeta = \cup_{n < \omega} F_{\zeta,n}$ with each $F_{\zeta,n}$ closed.
Of course, if $cf(\eta) \le \omega$ then it is trivial that $\eta$ is also an $F_\sigma$ set.

So, assume that  $cf(\eta) > \omega$ and, using that $S$ is non-reflecting, fix a closed unbounded subset
$C$ in $\eta$ with $0 \in C$ that is disjoint from $S$. For every $\gamma \in C$ let $\gamma^+$ denote the least
member of $C$ above $\gamma$.

Let us now define for each $n < \omega$ the set $F_{\eta,n}$ as follows:
$$F_{\eta,n} =  \bigcup \{F_{\gamma^+,n} \cap [\gamma, \gamma^+) : \gamma \in C\}\,.$$
Then it is obvious that every $F_{\eta,n}$ is closed in $X$,
i.e. does not contain the point $\kappa$ in its closure, and the union of the $F_{\eta,n}$'s equals $\eta$,
hence it is indeed an $F_\sigma$.

\end{proof}

Consistently, Theorem \ref{tm:nob} yields a very strong negative answer to the question of Bella and Spadaro 
in the non-compact case. Indeed, for example, in the constructible universe $L$, or in fact in any
set generic extension of $L$, there is a proper class of regular cardinals that
contain non-reflecting stationary sets of $\omega$-limits.
On the other hand, we do not know the answer to the following natural question:

\medskip

PROBLEM 2.
Is there a ZFC example of a countably tight Hausdorff (or regular, or Tychonov) space $X$ for which $t(X_\delta) > 2^\omega$?

\medskip

It is known that if there is a non-reflecting stationary set of $\omega$-limits in a regular cardinal $\kappa$,
then $\kappa$ is less than the first strongly compact cardinal $\lambda$, provided that it exists.
On the other hand, modulo some large cardinals, it is also known that there may be ZFC models in which such a $\lambda$ exists and
the cardinals $\kappa$ admitting a non-reflecting stationary set of $\omega$-limits are cofinal in $\lambda$, see e.g. \cite{BM}.
Consequently, our second result implies that the first one is in some sense sharp.

In this we shall use the following characterization of strongly compact cardinals, see \cite{K}:
The cardinal $\lambda$ is strongly compact iff for every set $A$ having cardinality at least $\lambda$
there is a $\lambda$-complete and fine free ultrafilter $\mathcal{U}$ on the set $[A]^{< \lambda}$.
That $\mathcal{U}$ is fine
means that for every element $a \in A$ we have $$\{B \in [A]^{< \lambda} : a \in B\} \in \mathcal{U}.$$

\begin{theorem}\label{tm:scp}
Let $\lambda$ be a strongly compact cardinal. Then for every topological space $X$ satisfying
$t(p,X) < \lambda$ for every $p \in X$ we have $t(X_\delta) \le \lambda$.
\end{theorem}

\begin{proof}
Assume that $A \subs X$ and $p \in X$ is a point such that for every $B \in [A]^{< \lambda}$
we have $p \notin \overline{B}^\delta$. We claim that then $p \notin \overline{A}^\delta$ as well.
This clearly implies $t(p,X_\delta) \le \lambda$, hence as $p \in X$ was arbitrary,
$t(X_\delta) \le \lambda$.

To prove our claim, let us fix for each set $B \in [A]^{< \lambda}$ a countable collection
$\{V_{B,n} : n < \omega\}$ of open neighborhoods of $p$ such that  $\bigcap \{V_{B,n} : n < \omega\} \cap B = \emptyset$.
This allows us to define the function $f_B : A \to \omega$ with the stipulation
$$f_B(x) = \min \{n : x \notin V_{B,n}\}.$$

Since $A$ clearly can be assumed to have cardinality $\ge \lambda$,
we may next fix a $\lambda$-complete and fine free ultrafilter on $\mathcal{U}$. Then we may define the function $f : A \to \omega$ with the stipulation
$$f(x) = n \,\Leftrightarrow \, \{B \in [A]^{< \lambda} : x \in B \mbox{ and } f_B(x) = n\} \in \mathcal{U}.$$
This makes sense because $\mathcal{U}$ is $\lambda$-complete and for every $x \in A$ we have
$$\{B \in [A]^{< \lambda} : x \in B\} \in \mathcal{U}.$$

Next we show that for every $n < \omega$ the closure in $X$ of the set $A_n = f^{-1}(n)$ misses
the point $p$. This clearly will imply that $p \notin \overline{A}^\delta$.

Now let $\mu = t(p,X) < \lambda$ and hence it will suffice to show that for every
$J \in [A_n]^\mu$ we have $p \notin \overline{J}$. As $\mathcal{U}$ is fine and $\lambda$-complete,
we have $$U = \{B \in [A]^{< \lambda} : J \subs B\} \in \mathcal{U}.$$
Moreover, this clearly implies that we also have
$$W = \{B \in U : \forall\, x \in J \, \big(f(x) = f_B(x)\big)\} \in \mathcal{U}.$$
But for any $B \in W$ we then have $J \cap V_{B,n} = \emptyset$, hence $p \notin \overline{J}$.
This completes our proof.

\end{proof}

{\bf Acknowledgements}. In the research on and preparation of this paper the second, third and fourth named authors
were supported by NKFIH grant no. K113047. The second author would also like to thank the support from the
Mathematics Department of UNC Charlotte and in particular Professor Alan Dow.


\newpage

\begin{thebibliography}{20}

\bibitem{BM}
Bagaria, Joan; Magidor, Menachem, 
On $\omega_1$-strongly compact cardinals. J. Symb. Log. 79 (2014), no. 1, 266-278.

\bibitem{BS}
A. Bella and S. Spadaro, Cardinal invariants for the $G_\delta$ topology
arXiv:1707.04871


\bibitem{D}
Alan Dow, An introduction to applications of elementary submodels to topology, Topology
Proc. 13 (1988), no. 1, 17--72.


\bibitem{J}
I. Juhász, Cardinal functions, ten years later, Math. Centre
Tract 123 (1980), Amsterdam.

\bibitem{CH*}
I.~Juh{\'a}sz, L.~Soukup, and Z.~Szentmikl{\'o}ssy,  What is left of
CH after you add Cohen reals ?, Topology and its Appl., 85 (1998),
pp~165--174.

\bibitem{K}
Kanamori, Akihiro, The higher infinite. Large cardinals in set theory from their beginnings. 
Perspectives in Mathematical Logic. Springer-Verlag, Berlin, 1994.

\end{thebibliography}
\end{document}